\documentclass[10pt]{amsart}
\usepackage{amsmath,amssymb,latexsym,soul,cite,mathrsfs,amsthm}

\usepackage{color,enumitem,graphicx}
\usepackage[colorlinks=true,urlcolor=blue,
citecolor=red,linkcolor=blue,linktocpage,pdfpagelabels,
bookmarksnumbered,bookmarksopen]{hyperref}
\usepackage[english]{babel}

\usepackage[left=2.9cm,right=2.9cm,top=2.8cm,bottom=2.8cm]{geometry}
\usepackage[hyperpageref]{backref}

\usepackage[colorinlistoftodos]{todonotes}
\makeatletter
\providecommand\@dotsep{5}
\def\listtodoname{List of Todos}
\def\listoftodos{\@starttoc{tdo}\listtodoname}
\makeatother

\numberwithin{equation}{section}
\def\dis{\displaystyle}

\newtheorem{lemma}{Lemma}

\newtheorem{theorem}{Theorem}
\newtheorem*{theorem*}{Theorem}

\newtheorem{remark}{Remark}

\title[A quasilinear Schr\"odinger-P0isson system]
{Existence and asymptotic behaviour of solutions\\ for a  quasi-linear Schr\"odinger-Poisson system \\
under a critical nonlinearity}

\author[G. M. Figueiredo]{Giovany M. Figueiredo}
\author[G. Siciliano]{Gaetano Siciliano}

\address[G. M. Figueiredo]{\newline\indent Universidade de Bras\'ilia-UNB
\newline\indent 
Departamento de Matem\'atica
\newline\indent
CEP: 70910-900, Bras\'ilia, DF, Brazil}
\email{\href{mailto:giovany@unb.br}{giovany@unb.br}}

\address[G. Siciliano]{\newline\indent Departamento de Matem\'atica
\newline\indent 
 Universidade de S\~ao Paulo 
\newline\indent 
Rua do Mat\~ao 1010,  05508-090 S\~ao Paulo, SP, Brazil }
\email{\href{mailto:sicilian@ime.usp.br}{sicilian@ime.usp.br}}

\thanks{The authors are partially supported by
CNPq, Capes,  FAPDF and Fapesp, Brazil.}

\begin{document}

\maketitle

\begin{abstract}
In this paper  we consider 
the following quasilinear Schr\"odinger-Poisson system  
$$
\left\{
\begin{array}[c]{ll}
- \Delta u +u+\phi u = \lambda f(x,u)+|u|^{2^{*}-2}u &\ \mbox{in }  \mathbb{R}^{3} \\
-\Delta \phi -\varepsilon^{4} \Delta_4 \phi = u^{2} &   \ \mbox{in } 
\mathbb{R}^{3},
\end{array}
 \right.
$$
depending on the two parameters $\lambda,\varepsilon>0$.

We first prove that, for  
$\lambda$  larger then a certain $\lambda^{*}>0$, there exists a solution
for every $\varepsilon>0$. Later, we study the asymptotic behaviour of these solutions
whenever $\varepsilon$ tends to zero, and we prove that they converge to
the solution of the Schr\"odinger-Poisson system associated.
\end{abstract}

\maketitle

\section{Introduction}

In \cite{BK,IKL} 
Kavian, Benmlih, Illner and Lange have attracted the attention on a 
new kind of elliptic system which, to the best of our knowledge, was never been considered before
in the mathematical literature, although the problem was  known among the physicists.
It seems it has been named {\sl quasi-linear Schr\"odinger-Poisson system}
and indeed is a generalization of the well-known Schr\"odinger-Poisson system.

This new system appears by studying a
quantum mechanical model of extremely small devices in semi-conductor nanostructures
taking into account quantum structure and   the longitudinal field oscillations during the beam propagation.
Indeed 
the intensity-dependent dielectric permittivity
 has the  form
 \begin{equation*}\label{eq:dielectr}
 {c}_{\textrm{diel}}(\nabla \phi)= 1+\varepsilon^{4}|\nabla \phi|^{2}, \quad \varepsilon>0 \ \textrm{ and constant}
 \end{equation*} 
 that is, it depends on the field itself.
We are considering the simplified case
of constant coefficient in $c_{\textrm{diel}}$ which corresponds to homogeneous medium.
 Here  $\phi$ is the electric field and  $\varepsilon$ appears to the power $4$ just for convenience.
 It seems this physical model and the corresponding equation of propagation
has appeared for the first time
in \cite{AA} where the authors proposed and discussed this new model 
(see also \cite{MRS}).

The main novelty with respect to the huge existing literature describing
beam propagation whenever $c_{\textrm{diel}}$ does not depend on the field, 
is that, from a mathematical point of view,  the equation of the electrostatic 
potential is not
linear, that is it is not the usual Poisson equation (or Gauss law in physical terms) given by $-\Delta\phi = u^{2}$.
Without entering in physical details here, the system one arrives by looking for standing waves
solutions is something of type
\begin{equation}\label{eq:Kavian}
\left\{
\begin{array}[c]{ll}
-\Delta u + \omega u+(\phi +\widetilde\phi )u = 0 & \medskip\\
 -\Delta \phi - \varepsilon^{4}\Delta_4 \phi =u^{2}-n^{*}&
\end{array}
 \right.
\end{equation}
where $u,\phi$ are the unknown functions (here $u$ represents the modulus of the wave function
and $\phi$ the electrostatic potential) and  $n^{*},\widetilde \phi :\mathbb R^{3}\to \mathbb R$ are given
data of the problem: they represent, respectively, the dopant density and  the effective external potential.
The operator $\Delta_{4}$ is the $4-$Laplacian, defined as $\Delta_4 u: = \mbox{div}(|\nabla u|^{2}\nabla u)$.
Indeed this is exactly the system introduced in the mathematical literature, as we said before, in the papers 
\cite{BK, IKL}.

Under minimal summability conditions on the data $n^{*}$ and $\widetilde\phi$
the authors in \cite{BK}, by means of minimization techniques,   proves the existence of ground state solutions and study its behaviour whenever $\varepsilon\to 0^{+}.$ Indeed they converges to the ground state solution
of the Schr\"odinger-Poisson system associated (that is, whenever $\varepsilon=0$ in \eqref{eq:Kavian}).

 A similar problem with periodicity conditions  is studied in \cite{IKL} where 
the existence of infinitely many  solutions normalized in $L^{2}$ by means of the Krasnoselkii genus
is proved.

Observe that the Schr\"odinger equation in the above system (i.e. the first equation) is linear in $u$.

\medskip


We point out that few other papers are known to treat this type of systems:  we revise now them here.

In the recent paper \cite{DLMZ},
Ding, Li, Meng and Zhuang
  deal with an asymptotically linear nonlinearity in the Schr\"odinger equation
and study the existence and the  behaviour of the ground state solution as  $\varepsilon\to 0^{+}$.
Again the solutions converge to the solution of the ``limit'' problem with $\varepsilon=0$.

 Illner, Lange, Toomire and Zweifel in \cite{ILTZ}  consider the quasilinear Schr\"odinger-Poisson system in the unitary cube under periodic 
boundary conditions and by using  Galerkin scheme, they prove global existence and uniqueness
of solutions.

In \cite{LY} Li and Yang study
the existence and uniqueness of a global mild solution to the initial boundary value problem in 
the one dimensional case


Finally  in the paper
\cite{PL}  of d'Avenia and  Pisani,
the Born-Infeld Lagrangian density
interacting with the Klein-Gordon equation is considered
They find infinitely many radial solutions in the subcritical case 
via the Symmetric Mountain Pass Theorem.
We cite this paper because the use of the Born-Infeld Lagrangian density for the electromagnetic field
(in place of the classical Maxwell Lagrangian density)
gives rise to the quasilinear equation for the electrostatic field. Indeed this will be our
approach to derive the system in the next Section.

\bigskip

%

It is clear that in theoretical analysis, numerical studies 
the most frequently used model for beam propagation assumes $c_\textrm{diel}(\nabla\phi) = 1$
which gives rise to the Poisson equation  $-\Delta \phi=u^{2}$ in the system.
 The advantage of working with the Poisson equation is that the solution is explicitly 
 given by the convolution $\phi^{\textrm {Poiss}}(u) = |\cdot|^{-1}*u^{2}$ (up to a multiplicative factor) so that many
 good properties of the solution are known;
 in particular the homogeneity
$\phi^{\textrm{Poiss}}(t u) = t^{2} \phi^{\textrm{Poiss}}(u), t\in \mathbb R.$
For the Schr\"odinger-Poisson system, 
 the existing literature is so huge that is almost impossible to give a satisfactory
list of papers. As a matter of fact, the main difficult dealing with the quasilinear Poisson equation
of type
$$-\Delta \phi-\Delta_{4}\phi = u^{2}$$
is due exactly to the lack of good properties for  the solution.

\bigskip

Coming back to the present paper, our aim is to study a system similar to 
\eqref{eq:Kavian} where the Schr\"odinger equation has a   critical nonlinearity;
more specifically, we are concerning here with the following system
\begin{equation}\label{eq:P}\tag{$P_{\lambda,\varepsilon}$}
\left\{
\begin{array}[c]{ll}
-\Delta u + u+\phi u = \lambda f(x,u)+|u|^{2^{*}-2}u & \ \mbox{in} \ \
\mathbb{R}^{3}, \medskip\\
 -\Delta \phi - \varepsilon^{4}\Delta_4 \phi =u^{2}& \  \mbox{in} \ \ \mathbb{R}^{3},
\end{array}
 \right.
\end{equation}

\noindent where 
\begin{itemize}
\item $\lambda>0$ and $\varepsilon>0$ are   parameters,
\item  $2^{*}=6$ is the critical Sobolev exponent in dimension 3,
\item $f:\mathbb{R}^{3}\times\mathbb{R} \rightarrow\mathbb{R}$ is a continuous
function that satisfies the following assumptions \medskip
\end{itemize}

\begin{enumerate}[label=(f\arabic*),ref=f\arabic*,start=0]
\item\label{f_{0}} $f(x,t)=0$ for $t\leq0$, \medskip
\item \label{f_{1}}$ \lim_{t \rightarrow 0}\dis \frac{f(x,t)}{t}=0,$ uniformly on $x\in \mathbb R^{3}$, \medskip
\item\label{f_{2}}there exists $q \in (2, 2^{*})$ verifying
$\dis\lim_{t \rightarrow +\infty}\frac{f(x,t)}{t^{q-1}}=0$ uniformly on $x\in\mathbb R^{3}$, \medskip 

\item\label{f_{3}} there exists $\theta \in(4,2^{*})$ such that
$$
0<\theta F(x,t)=\theta\int^{t}_{0}f(x,s)ds \leq tf(x,t), \quad
\mbox{for all}\,\,\, x\in \mathbb R^{3} \ \mbox{and} \ \ t>0.
$$
\end{enumerate}
A typical example of  function satisfying the above conditions is
$$
f(x,t)= \dis\sum^{k}_{i=1}C_{i}(x)t_{+}^{q_{i}-1}
$$
with $k \in \mathbb{N}$, $4<q_{i}<2^{*}$, $C_{i}$ bounded and positive functions 
and $t_{+}=\dis\max\{t,0\}$.

\bigskip

Before introducing the notion of solution, we establish few basic standard notations.

For $p\in[1,+\infty], L^{p}(\mathbb R^{3})$ is the usual Lebesgue space
with norm $|u|_{p}$.

We denote with $H^{1}(\mathbb R^{3})$ the usual Sobolev space endowed with scalar product 
and norm given by
$$\langle u, v\rangle_{H^{1}}:=\int_{\mathbb R^{3}} \nabla u\nabla v +\int_{\mathbb R^{3}} uv,\qquad \|u\|_{H^{1}}:= \langle u, u\rangle^{1/2}.$$

For $p\geq2, D^{1,p}(\mathbb R^{3})$ is the Banach space defined as the completion of the test functions $C^{\infty}_{c}(\mathbb R^{3})$
with respect to the  $L^{p}-$ norm of the gradient. We define
$$X:=D^{1,2}(\mathbb R^{3})\cap D^{1,4}(\mathbb R^{3})$$
which  is a Banach space under the norm
$$\|\phi\|_{X}:=|\nabla \phi|_{2} + |\nabla \phi|_{4}.$$

As a final convention, whenever we are understanding the Lebesgue measure $dx$ in integrals, it will be always omitted;
otherwise we will write explicitly the measure.

\bigskip

The natural functional spaces in which find the solutions of \eqref{eq:P} are:
$$
u\in H^{1}(\mathbb R^{3}), \quad
\phi \in X.
$$
By a solution of \eqref{eq:P} we mean a pair 
$(u_{\lambda,\varepsilon}, \phi_{\lambda,\varepsilon})\in H^{1}(\mathbb R^{3})\times X$ such that
\begin{equation}\label{eq:ws1}
\forall v\in H^{1}(\mathbb R^{3}): \quad \int_{\mathbb R^{3}} \nabla u_{\lambda,\varepsilon} \nabla v+\int_{\mathbb R^{3}} u_{\lambda,\varepsilon} v +\int_{\mathbb R^{3}} \phi_{\lambda,\varepsilon} u_{\lambda,\varepsilon}v
=\lambda \int_{\mathbb R^{3}} f(x,u_{\lambda,\varepsilon}) v+\int_{\mathbb R^{3}} |u_{\lambda,\varepsilon}|^{2^{*}-2}u_{\lambda,\varepsilon}v\\
\end{equation}
\begin{equation}\label{eq:ws2}
\forall \xi \in X:\quad  \int_{\mathbb R^{3}} \nabla \phi_{\lambda,\varepsilon} \nabla \xi
+\varepsilon^{4}\int_{\mathbb R^{3}} |\nabla \phi_{\lambda,\varepsilon}|^{2}\nabla  \phi_{\lambda,\varepsilon}\nabla \xi = \int_{\mathbb R^{3}} \xi u^{2} .
\end{equation}

\medskip

The main results of this paper are the following.
\begin{theorem} \label{teorema1}
Assume that conditions \eqref{f_{0}}-\eqref{f_{3}} hold.
Then, there exists $\lambda^{*}>0$, such that
$$\forall \lambda\geq\lambda^{*}, \varepsilon>0: \textrm{problem \eqref{eq:P} admit a solution 
$(u_{\lambda,\varepsilon}, \phi_{\lambda,\varepsilon})\in H^{1}(\mathbb R^{3})\times X$}.$$
Moreover  $\phi_{\lambda,\varepsilon}, u_{\lambda,\varepsilon}$ are nonnegative, of Mountain
Pass type
and for every fixed $\varepsilon>0$:
\begin{itemize}
\item[1.] $\lim_{\lambda\rightarrow + \infty} \|u_{\lambda,\varepsilon}\|_{H^{1}} = 0$,
\item[2.] $\lim_{\lambda\rightarrow + \infty} \|\phi_{\lambda,\varepsilon}\|_{X} =0$,
\item[3.] $\lim_{\lambda\rightarrow + \infty} |\phi_{\lambda,\varepsilon}|_{\infty}=0$.
\end{itemize}
%
\end{theorem}
Actually, except for the limit in $3.$, the Theorem  also holds for $\varepsilon=0$
by replacing $X$ with $D^{1,2}(\mathbb R^{3})$.

We study also the behaviour with respect to $\varepsilon$ of the solutions 
given in Theorem \ref{teorema1},
indeed we prove they converge to the solution of the Schr\"odinger-Poisson 
system.
\begin{theorem}\label{teorema2}
	Assume that conditions \eqref{f_{0}}-\eqref{f_{3}} hold. Let $\lambda^{*}>0$
	be the one given in Theorem \ref{teorema1}  and 
	$\overline\lambda\geq\lambda^{*}$ be fixed.
	Let $\{(u_{\overline\lambda,\varepsilon}
	, \phi_{\overline\lambda,\varepsilon})\}_{\varepsilon>0}$ be the solutions given above
	in correspondence of such fixed $\overline \lambda$. 
	 Then
	\begin{itemize}
		\item[1.] $\lim_{\varepsilon\to0^{+}} u_{\overline\lambda,\varepsilon} = u_{\overline\lambda,0}$ in $ H^{1}(\mathbb R^{3})$,
		\item[2.]  $\lim_{\varepsilon\to0^{+}} \phi_{\overline\lambda,\varepsilon} = \phi_{\overline\lambda,0}$ in $ D^{1,2}(\mathbb R^{3})$,
	\end{itemize}
	where $(u_{\overline\lambda,0}, \phi_{\overline\lambda,0})\in H^{1}(\mathbb R^{3})\times D^{1,2}(\mathbb R^{3})$ is a  positive solution, of Mountain Pass type of the Schr\"odinger-Poisson system
	\begin{equation}\label{eq:SP}
	\left\{
	\begin{array}[c]{ll}
	-\Delta u + u+\phi u = \overline\lambda f(x,u)+|u|^{2^{*}-2}u & \ \mbox{in} \ \
	\mathbb{R}^{3}, \medskip\\
	-\Delta \phi = u^{2}& \  \mbox{in} \ \ \mathbb{R}^{3}.
	\end{array}
	\right.
	\end{equation}
\end{theorem}

The important point of Theorem \ref{teorema1}
is the vanishing of the solutions whenever $\lambda$
is larger and larger. Moreover, thanks to a Moser iteration scheme,
we get 
$u_{\lambda,\varepsilon},\phi_{\lambda}\in L^{\infty}(\mathbb R^{3})$
 This allow us to treat also the supercritical case, hence a problem of type 
\begin{equation}\label{eq:Plambdap}
\left\{
\begin{array}[c]{ll}
-\Delta u + u+\phi u = \lambda f(x,u)+|u|^{p-2}u & \ \mbox{in} \ \
\mathbb{R}^{3}, p>2^{*},\medskip\\
 -\Delta \phi - \varepsilon^{4}\Delta_4 \phi =u^{2}& \  \mbox{in} \ \ \mathbb{R}^{3},
\end{array}
 \right.
\end{equation}
under the same assumptions on $f$.
More explicitly, as a consequence of Theorem \ref{teorema1} we have the following
\begin{theorem}\label{teorema3}
Theorem \ref{teorema1} and Theorem \ref{teorema2} hold also for problem \eqref{eq:Plambdap}.
\end{theorem}

Our approach in proving Theorem \ref{teorema1} is variational. Indeed a suitable functional
can be defined whose critical points are exactly the solutions of \eqref{eq:P}.
Hence the meaning of ``Mountain Pass solution'' will be clear.

In proving our results,  we have to manage with various difficulties.
Firstly, the fact that the problem is in the whole $\mathbb R^{3}$
and no symmetry conditions on the solutions and on the datum $f$
are imposed (as e.g. in \cite{PL}); even more  we are in the critical case,
then there is a clear lack of compactness. We are able to overcome this difficulty thanks
to   the Concentration Compactness of Lions (see \cite{Lio2}) and taking advantage of the parameter $\lambda$.

Secondly, we have to face with the fact that the solution in the second equation of \eqref{eq:P},
which is quasilinear, has not an explicit formula, neither has homogeneity properties. 
To circumvent this last difficulty, a suitable truncation (already introduced in \cite{K})  is used
in front of the ``bad'' part of the functional.
This type of truncation is also used in \cite{AdP} to treat the classical
Schr\"odinger-Poisson  problem
under a general nonlinearity of Berestycki-Lions type.

\medskip

Our contribution in this paper is then to give a better understanding
of this intriguing problem, especially to see as the truncation argument, which appears for the first time
for a quasilinear Schr\"odinger-Poisson system,
is useful to deal with a piece of the functional which has not good properties.

Note that when $\varepsilon=0$, that is in the case of the Schr\"odinger-Poisson system,
Theorem \ref{teorema1} gives the result of Zhao and Zhao \cite{ZZ},
which indeed concerns with a slightly different nonlinearity 
of type $g(x,u)= \mu Q(x)|u|^{q-2}u+K(x)|u|^{2^{*}-2}u$.
However with slight changes,  our theorems also hold
if in front of the critical nonlinearity there is coefficient $K(x)$ as in \cite{ZZ}.

Moreover, as a byproduct of our Theorem \ref{teorema3} we deduce the existence
of a positive solution (for $\lambda$ large) for the Schr\"odinger-Poisson system
even in presence of  a supercritical nonlinearity, fact that we were not able to find 
in the literature.

\medskip

The paper is organized as follows.

 In Section \ref{sec:derivation}
we deduce the set of equations we are going to study.
Indeed, differently from the paper of Benmilh and Kavian \cite{BK}, we deduce the equations
under study in the framework of the Abelian Gauge Theories by considering the interaction
of the Schr\"odinger equation with the Maxwell equation described by the 
Born-Infeld Lagrangian which is the
second order approximation of the classical Maxwell Lagrangian.

 In Section \ref{sec:VF} the variational framework of the problem is introduced.
We study some properties of the second equation in the system and define the functional $J_{\lambda,\varepsilon}$
whose critical points will be the solution of the system.

In Section \ref{sec:trunc} we introduce the truncation in the original functional $J_{\lambda,\varepsilon}$. 
This will help to deal with the lack of properties
of the solution of the second equation, in contrast to the case of the Schr\"odinger-Poisson system.

In Sections \ref{sec:final}, \ref{sec:T2} and \ref{sec:T3}  the proof of 
Theorem \ref{teorema1}, \ref{teorema2} and \ref{teorema3}, respectively, is given.


\section{Derivation of the system}\label{sec:derivation}
Let us spend few words in this section on the physical derivation
of system \eqref{eq:P} in the framework of Abelian Gauge Theory.

Our starting point is the Lagrangian of the nonlinear Schr\"odinger equation.
Indeed it is well known that the Euler Lagrange equation of the Lagrangian density
\begin{eqnarray*} 
\mathcal L_{\textrm{S}}(\psi) = i\hbar  \overline \psi \partial_{t}\psi-\frac{\hbar^{2}}{2m}|\nabla \psi|^{2}+G(x,|\psi|)
\end{eqnarray*}
is exactly the nonlinear Schr\"odinger equation.
Here $G(x,|\psi|)$ is a suitable nonlinearity depending on the physical model.

The interaction of the wave function $\psi$ with   the electromagnetic field
generated by its motion, is described by means of the {\sl covariant derivative}
in the framework of the Abelian Gauge Theory. In Physics this is known also
as {\sl minimal coupling rule} and, practically, consists in substituting the ordinary derivative
in $\mathcal L_{\textrm{S}}$ with the new operators (the {\sl covariant}, or {\sl Wayl  derivatives}):
$$\partial_{t} \rightarrow\partial_{t}+\frac{iq}{\hbar}\phi,\quad \nabla \rightarrow \nabla-\frac{iq}{\hbar}\mathbf A$$
where $\phi$ and $\mathbf A$ are the gauge potentials of the electromagnetic field, that is
$$\mathbf E = -\nabla \phi -\partial_{t}\mathbf A,\quad \mathbf B= \nabla \times\mathbf A,$$
$q$ is the electric charge and $\hbar$ the normalized Plank constant.
In this way one obtains from $\mathcal L_{S}$ the  Lagrangian density of the interaction
\begin{eqnarray*}
\mathcal L_{\textrm{Int}}(\psi,\phi,\mathbf A) = i\hbar \overline \psi \partial_{t} \psi - q \phi |\psi|^{2} 
-\frac{\hbar^{2}}{2m}\left|\nabla \psi -\frac{iq}{\hbar}\mathbf A\psi\right|^{2} +G(x,|\psi|).
\end{eqnarray*}
It is convenient to write the wave function in polar form, i.e.
 $\psi(x,t) = u(x,t)e^{iS(x,t)/\hbar}$
 with $u, S: \mathbb R^{3}\times\mathbb R\to \mathbb R$. Then the  Lagrangian density
 of the interaction takes the form
\begin{eqnarray*}\label{eq:interazione}
\mathcal L_{\textrm{Int}}(u,S,\phi,\mathbf A) = i\hbar u \partial_{t} u -\frac{\hbar^{2}}{2m}|\nabla u|^{2}-
\left( \partial_{t}S +q\phi +\frac{1}{2m}\left|\nabla S -q\mathbf A\right|^{2}\right)u^{2}
 +G(x,u).
\end{eqnarray*}
However this is not the total Lagrangian density of the system since the e.m. field (and then $\phi$
and $\mathbf A$) is an unknown, hence also the 
Lagrangian density of the e.m. has to be considered.

The existing literature concerning the Schr\"odinger-Maxwell system, mainly
consider the usual classical Lagrangian density of Maxwell:
$$\mathcal L_{\textrm{M}}(\phi,\mathbf A) = \frac{1}{8\pi}\left( |\mathbf E|^{2} -|\mathbf B|^{2}\right)
=\frac{1}{8\pi} \left(|\nabla \phi+\partial_{t}\mathbf A|^{2} - |\nabla \times \mathbf A|^{2}\right).$$
Here we use  the Lagrangian density of the Born-Infeld theory, that is
$$\mathcal L_{\textrm{BI}}= \frac{1}{4\pi}\left[ \frac{1}{2} \left(|\mathbf E |^{2} - |\mathbf B|^{2}\right) +\frac{\beta}{4}
\left(|\mathbf E |^{2} - |\mathbf B|^{2}\right)^{2} \right], \quad \beta>0.$$
In this way the total Lagrangian density, which describes the dynamic of the motion of the matter field
$\psi$ and the e.m. field $(\mathbf E, \mathbf B)$, is given by
\begin{eqnarray*}\label{eq:}
\mathcal L_{\textrm{tot}}(u,S,\phi,\mathbf A) &=&  \mathcal L_{\textrm{Int}}(u,S,\phi,\mathbf A) +\mathcal L_{\textrm{BI}}(\phi,\mathbf A ) \\
&=&
i\hbar u \partial_{t} u -\frac{\hbar^{2}}{2m}|\nabla u|^{2}-
\left( \partial_{t}S +q\phi +\frac{1}{2m}\left|\nabla S -q\mathbf A\right|^{2}\right)u^{2}
 +G(x,u) \\
 &+& \frac{1}{4\pi}\left[ \frac{1}{2} \left(|\mathbf E |^{2} - |\mathbf B|^{2}\right) +\frac{\beta}{4}
\left(|\mathbf E |^{2} - |\mathbf B|^{2}\right)^{2} \right].
\end{eqnarray*}

The Euler Lagrange equations of this Lagrangian (that is, by making the variations
with respect to $u,S,\phi,\mathbf A$) are easily computed and are
\begin{equation*}\label{eq:}
-\dis\frac{\hbar^{2}}{2m}\Delta u +\left(\partial_{t}S +q \phi +\frac{1}{2m}|\nabla S-q\mathbf A|^{2} \right) u = g(x, u)  
\end{equation*}
\begin{equation*}
\partial_{t}u^{2} +\dis\frac{1}{m} \nabla \cdot \left[(\nabla S -q\mathbf A)u^{2}\right]=0 
\end{equation*}
\begin{equation*}
-\nabla\cdot\Big( Z_{\mathbf A, \phi} (\nabla \phi+\partial_{t}\mathbf A)\Big) =4\pi qu^{2} 
\end{equation*}
\begin{equation*}
\partial_{t} \Big( Z_{\mathbf A, \phi}(\partial_{t}\mathbf A+\nabla\phi)\Big) + 
\nabla\times \Big( Z_{\mathbf A, \phi} \nabla \times\mathbf A \Big) = 4\pi q(\nabla S-q\mathbf A)u^{2}
\end{equation*}
where we have set, for brevity, 
$$ Z_{\mathbf A, \phi}:= 1+\beta|\mathbf A_{t}+\nabla \phi|^{2} -\beta|\nabla \times \mathbf A|^{2}$$
and $g(x,s)=\partial_{s} G(x,s).$

An interesting physical situation is that of 
standing waves in the purely electrostatic case which appears when we look for solutions of type 
$$u(x,t)=u(x), \quad S(x,t)=\omega\hbar t,\quad  \phi(x,t) = \phi(x), \quad \mathbf A(x,t)=\mathbf 0$$
which gives rise to wave functions of type $\psi(x,t) =u(x)e^{i \omega t}$.
In this case the above set of equations is reduced to 
\begin{equation}\label{eq:constantes}
\left\{
\begin{array}[c]{ll}
-\dis\frac{\hbar^{2}}{2m} \Delta u +\omega u+q\phi u = g(x,u) & \ \mbox{in} \ \
\mathbb{R}^{3}, \medskip \\
 -\nabla \cdot\left( \nabla \phi + \beta|\nabla \phi|^{2}\nabla \phi \right)= 4\pi qu^{2}& \  \mbox{in} \ \ \mathbb{R}^{3}.
\end{array}
 \right.
\end{equation}
Observe that up to change $\phi$ with $-\phi$, we can assume without lost of
generality that $q>0$. 
By ``normalizing'' the constants
$$\frac{\hbar^{2}}{2m} = \omega =q=4\pi =1,$$
and setting 
$$\beta=\varepsilon\quad \text{and} \quad g(x,u)=\lambda f(x,u)+|u|^{2^{*}-2}u$$ 
problem \eqref{eq:constantes} becomes   exactly  problem \eqref{eq:P}.

\section{The variational framework}\label{sec:VF}

We begin by saying that   that  the single equation
$$-\Delta \phi -\beta\Delta_{4}\phi = \rho \quad (\beta>0)$$
has been very studied in the mathematical literature,
since it falls down into the class of equations involving the $p\&q$ Laplacian.
In particular in \cite{FOP}, where the authors
study the case in which the distribution $\rho$ is a Dirac delta or an $L^{1}$ function,
 it is shown that there is the continuous embedding 
$$X\hookrightarrow L^{\infty}(\mathbb R^{3})$$
(see \cite[Proposition 8]{FOP}). 
As a consequence,  the solutions $\phi_{\lambda,\varepsilon}$ given in Theorem \ref{teorema1}
will be automatically in $L^{\infty}(\mathbb R^{3})$; moreover once we prove that
 $\lim_{\lambda\to+\infty}\|\phi_{\lambda,\varepsilon}\|_{X} =0$, then we have for free that
$\lim_{\lambda\to+\infty}|\phi_{\lambda,\varepsilon}|_{\infty} =0$.

 \bigskip

We have now a first variational principle; indeed, it is easy to see that the critical points of the $C^{2}$ functional
\begin{equation}\label{eq:F}
\mathcal J_{\lambda,\varepsilon}(u,\phi) = \frac{1}{2}\|u\|_{H^{1}}^{2}+\frac{1}{2}\int_{\mathbb R^{3}} \phi u^{2} -\lambda\int_{\mathbb R^{3}} F(x,u) -\frac{1}{2^{*}}\int_{\mathbb R^{3}} |u|^{2^{*}}
-\frac{1}{4}\int_{\mathbb R^{3}} |\nabla \phi|^{2} -\frac{\varepsilon^{4}}{8}\int_{\mathbb R^{3}} |\nabla \phi|^{4}
\end{equation}
on $H^{1}(\mathbb R^{N})\times X $
are exactly the weak solutions of \eqref{eq:P}, according to \eqref{eq:ws1} and \eqref{eq:ws2}.
However since this functional $\mathcal J_{\lambda,\varepsilon}$ is strongly indefinite,
we adopt a reduction procedure which is successfully used 
with the ``classical'' Schr\"odinger-Poisson system.

\subsection{Study of the quasilinear Schr\"odinger-Poisson equation}
Let us consider for convenience the following general problem
\begin{equation}\label{eq:geral}
-\Delta \phi -\Delta_{4}\phi = g \in X'.
\end{equation}
This problem has a unique solution $\phi_{g}$.
This follows by the fact that the $C^{1}$ functional
$$\phi \in X \longmapsto \frac{1}{2}\int_{\mathbb R^{3}} |\nabla \phi|^{2} + \frac{1}{4}\int_{\mathbb R^{3}} |\nabla \phi|^{4} - g[\phi] \in \mathbb R$$
is strictly convex, coercive and weakly lower semicontinuous; hence
possess a unique critical point, denoted with $\phi(g)$, which is a minimum\and a solution of 
\eqref{eq:geral}.
Alternatively, the existence of a unique solution $\phi(g)$ can be deduced by using
the Minty-Browder's Theorem \cite[Teorema V. 15]{brezis}, since
the operator $\textrm{T}:X\rightarrow X'$ defined by duality by
$$
\langle \textrm{T}(\phi),\xi\rangle:=\int_{\mathbb{R}^{3}} \nabla \phi\nabla \xi 
+\int_{\mathbb{R}^{3}} |\nabla \phi|^{2}\nabla\phi\nabla \xi ,
$$
 is continuous,  strictly monotone and coercive.
Anyway, from unicity result,  it is well defined the solution operator
\begin{equation*}\label{eq:Phi}
\Phi:X' \to X, \ \ \Phi(g):=\phi(g)
\end{equation*}
associated to equation \eqref{eq:geral}.

In the next result, we show that the solution operator $\Phi$ is continuous. 
\begin{lemma}\label{solutionoperator}
Let $g_{n}\to g$ in $X'$. Then, 
we have
$$\int_{\mathbb R^{3}} |\nabla \phi(g_{n})|^{2} \to \int_{\mathbb  R^{3}} |\nabla \phi(g)|^{2}, \quad
 \int_{\mathbb  R^{3}} |\nabla \phi(g_{n})|^{4}\to \int_{\mathbb  R^{3}} |\nabla \phi(g)|^{4}$$
and consequently
$$\phi(g_{n})\to \phi(g) \text{ in } \  L^{\infty}(\mathbb R^{3}).$$
In particular the operator $\Phi$ is continuous.
\end{lemma}
\begin{proof}
By assumptions for every $w\in X$,
$$
\displaystyle\int_{\mathbb{R}^{3}}\nabla \phi(g_{n}) \nabla w + \displaystyle\int_{\mathbb{R}^{3}}|\nabla \phi(g_{n})|^{2} \nabla \phi(g_{n}) \nabla w
= g_n [w]
$$
and 
$$
\displaystyle\int_{\mathbb{R}^{3}}\nabla \phi(g) \nabla w + \displaystyle\int_{\mathbb{R}^{3}}|\nabla \phi(g)|^{2} \nabla \phi(g) \nabla w= g [w].
$$
We conclude that 
$$
\int_{\mathbb{R}^{3}}\nabla \phi(g_{n}) \nabla w + \int_{\mathbb{R}^{3}}|\nabla \phi(g_{n})|^{2} \nabla \phi(g_{n}) \nabla w
- \displaystyle\int_{\mathbb{R}^{3}}\nabla \phi(g) \nabla w - \displaystyle\int_{\mathbb{R}^{3}}|\nabla \phi(g)|^{2} \nabla \phi(g) \nabla w = o_n(1).
$$
Considering $w=\phi(g_{n})-\phi(g)$, we derive
\begin{eqnarray*}
\displaystyle\int_{\mathbb{R}^{3}}|\nabla \phi(g_{n})-\nabla \phi(g) |^{2} + \int_{\mathbb{R}^{3}}\left(|\nabla \phi(g_{n})|^{2}\nabla \phi(g_{n}))- |\nabla \phi(g)|^{2}\nabla \phi(g)\right)(\nabla \phi(g_{n}) -\nabla \phi(g))= o_n(1)
\end{eqnarray*}
and then by the Simon inequality there exists $C>0$ such that
\begin{multline*}
\int_{\mathbb{R}^{3}}|\nabla \phi(g_{n})-\nabla \phi(g) |^{2}+ C\int_{\mathbb{R}^{3}}|\nabla \phi(g_{n})-\nabla \phi(g) |^{4}
\leq  \\ \int_{\mathbb{R}^{3}}|\nabla \phi(g_{n})-\nabla \phi(g) |^{2} + \int_{\mathbb{R}^{3}}\left(|\nabla \phi(g_{n})|^{2}\nabla \phi(g_{n})- \nabla \phi(g)|^{2}\nabla \phi(g)\right)(\nabla \phi(g_{n}) -\nabla \phi(g))= o_n(1).
\end{multline*}
which concludes the proof.
\end{proof}
 
 Of course all that we have seen here also holds for the problem
 $$-\Delta \phi -\varepsilon^{4}\Delta_{4}\phi = g \in X',$$
by considering the map $\Phi_{\varepsilon}$, for every $\varepsilon>0$.

\medskip

\subsection{The reduction argument}
Let us consider now  a particular case of the previous subsection.
Let $u\in H^{1}(\mathbb R^{3})$ and note that $u^{2}\in X'$  in the sense that
the map
\begin{equation*}\label{eq:}
g_{u^{2}}:\phi \in X \longmapsto \int _{\mathbb R^{3}}\phi u^{2}\in \mathbb R
\end{equation*}
is linear and continuous.
Then for every $u\in H^{1}(\mathbb R^{3})$ fixed, there exists a unique element in $X$,
that we denote with $\phi_{\varepsilon}(u)$, such that
\begin{equation}\label{eq:2eq}
-\Delta \phi_{\varepsilon}(u) -\varepsilon^{4}\Delta_{4}\phi_{\varepsilon}(u) = u^{2} \quad \text{ in }\mathbb R^{3}.
\end{equation}
 In the remaining of the paper, $\phi_{\varepsilon}(u)$ will always denote the unique solution of \eqref{eq:2eq},
which, {\sl en passant}, satisfies
\begin{equation}\label{eq:sostituicao}
\int_{\mathbb R^{3}} |\nabla \phi_{\varepsilon}(u)|^{2} + \varepsilon^{4}\int_{\mathbb R^{3}}|\nabla \phi_{\varepsilon}(u)|^{4} = \int_{\mathbb R^{3}} \phi_{\varepsilon}(u) u^{2}.
\end{equation}
In particular we have the following.
\begin{lemma}\label{lem:facile}
If $\{u_{n}\}$ converges to $u$ in $L^{12/5}(\mathbb R^{3})$,
 then, for every fixed $\varepsilon>0$: \medskip
 \begin{itemize}
 \item[(a)] $\displaystyle \lim_{n\to+\infty}\int_{\mathbb R^{3}} |\nabla \phi_{\varepsilon}(u_{n})|^{2} = \int_{\mathbb R^{3}} |\nabla \phi_{\varepsilon}(u)|^{2}$, \medskip
\item[(b)] $ \displaystyle\lim_{n\to+\infty}\int_{\mathbb R^{3}} |\nabla \phi_{\varepsilon}(u_{n})|^{4} =\int_{\mathbb R^{3}} |\nabla \phi_{\varepsilon}(u)|^{4}$, \medskip
\item[(c)]$ \displaystyle\lim_{n\to+\infty}\int_{\mathbb R^{3}}\phi_{\varepsilon}(u_{n})u_{n}^{2} = \int_{\mathbb R^{3}}\phi_{\varepsilon}(u) u^{2} $,\medskip
\item[(d)] $\displaystyle\lim_{n\to+\infty}\phi_{\varepsilon}(u_{n})= \phi_{\varepsilon}(u)$ in $L^{\infty}(\mathbb R^{3}).$
 \end{itemize}
\end{lemma}

\begin{proof}
Under our assumptions we have,
$$\|g_{u_{n}^{2}} - g_{u^{2}}\| = \sup_{\|\phi\|_{X}=1} \Big| \int_{\mathbb R^{3}} \phi (u_{n}^{2} - u^{2})\Big|\leq |\phi|_{2^{*}} |u_{n}^{2} - u^{2}|_{(2^{*})'}\leq C |u_{n}^{2} - u^{2}|_{6/5}\to 0.$$
Then we can apply Lemma \ref{solutionoperator} and  conclude the proof.
\end{proof}



We introduce  the map
$$\Phi_{0}: u\in H^{1}(\mathbb R^{3})\mapsto \phi_{0}(u)\in D^{1,2}(\mathbb R^{3}).$$
Many properties of this map are well known, 
in particular  (a) and (c) of Lemma \ref{lem:facile}.

The next result is a consequence of the fact that $\mathcal J_{\lambda,\varepsilon}$ is $C^{2}$ 
and the Implicit Function Theorem.
The arguments used to prove Lemma \ref{lem:1} and Lemma \ref{lem:2} 
are exactly the same as  in \cite{BF}
for the Schr\"odinger-Poisson system
(that is for the map $\Phi_{0}$ defined above), or \cite{BFKleinG} for the Klein-Gordon-Maxwell
system.

\begin{lemma}\label{lem:1}
For all $\varepsilon>0$, let $G_{\Phi_{\varepsilon}}$ be the graph of the map $\Phi_{\varepsilon}: u\in H^{1}(\mathbb R^{3})\mapsto \phi_{\varepsilon}(u)\in X$.
Then
$$G_{\Phi_{\varepsilon}}=\left\{(u,\phi)\in H^{1}(\mathbb R^{3})\times X: \partial_{\phi}\mathcal J_{\lambda,\varepsilon}(u,\phi) = 0\right\}.$$
Moreover 
$$\Phi_{\varepsilon}\in C^{1}(H^{1}(\mathbb R^{3}); X ).$$
\end{lemma}

In view of this, the functional (recall \eqref{eq:sostituicao})
\begin{eqnarray*}
J_{\lambda,\varepsilon}(u)&:=&\mathcal J_{\lambda,\varepsilon}(u,\phi_{\varepsilon}(u))\\
&=&
\frac 12 \|u\|^{2}+ \frac 14 \displaystyle\int_{\mathbb{R}^{3}}|\nabla \phi_{\varepsilon}(u)|^{2}  + \frac{3\varepsilon^{4}}{8} \displaystyle\int_{\mathbb{R}^{3}}|\nabla \phi_{\varepsilon}(u)|^{4}  
-\lambda\displaystyle\int_{\mathbb{R}^{3}}F(x,u) - \frac{1}{2^{*}}\displaystyle\int_{\mathbb{R}^{3}}|u|^{2^{*}}
\end{eqnarray*}
is of class $C^{1}$ and in particular we have
\begin{eqnarray*}\label{eq:}
J_{\lambda,\varepsilon}'(u)[v] &= &\partial_{u} \mathcal J_{\lambda,\varepsilon}(u,\phi_{\varepsilon}(u))[v]+
\partial_{\phi}\mathcal J_{\lambda,\varepsilon}(u,\phi_{\varepsilon}(u))\circ \Phi_{\varepsilon}'(u) [v] \\ 
&=&\partial_{u} \mathcal J_{\lambda,\varepsilon}(u,\phi_{\varepsilon}(u))[v].
\end{eqnarray*}
Then by \eqref{eq:F} we have
$$J_{\lambda,\varepsilon}'(u)[v]  = \int_{\mathbb R^{3}} \nabla u \nabla v+\int_{\mathbb R^{3}} uv+\int_{\mathbb R^{3}}\phi_{\varepsilon}(u)uv
-\lambda\int_{\mathbb R^{3}} f(x,u)v -\int_{\mathbb R^{3}}|u|^{2^{*}-2}uv$$
from which a second variational principle holds:
\begin{lemma} \label{lem:2}
Let $\lambda,\varepsilon>0 $ be fixed. The following statements are equivalent:
\begin{itemize}
\item[(i)] the pair $(u_{\lambda,\varepsilon},\phi_{\lambda,\varepsilon})\in H^{1}(\mathbb R^{3})\times X$ is a critical point of $\mathcal J_{\lambda,\varepsilon}$ (i.e. $(u_{\lambda,\varepsilon},\phi_{\lambda,\varepsilon})$ is a solution of \eqref{eq:P}),\medskip
\item[(ii)] $u_{\lambda,\varepsilon}$ is a critical point of $J_{\lambda,\varepsilon}$ and $\phi_{\lambda,\varepsilon} =\phi_{\varepsilon}(u_{\lambda}) 
$.
\end{itemize}
\end{lemma}

\medskip

The functional $J_{\lambda,\varepsilon}$ of the unique variable $u$  obtained by $\mathcal J_{\lambda,\varepsilon}$
is usually called the {\sl reduced functional}.

\medskip

In view of Lemma \ref{lem:2},
 the critical points of $J_{\lambda,\varepsilon}$ satisfy the equation
\begin{equation}\label{eq:equation}
- \Delta u +u+\phi_{\varepsilon}(u) u = \lambda f(x,u)+|u|^{2^{*}-2}u \ \ \  \mbox{in} \ \ \
\mathbb{R}^{3}, 
\end{equation}
which is the equation we are going to consider in the following.

\medskip

It will be convenient to  introduce the functional
\begin{equation*}\label{eq:I}
I_{\varepsilon}:u\in H^{1}(\mathbb R^{3})\longmapsto \dis\frac{1}{4}\int_{\mathbb R^{3}}|\nabla \phi_{\varepsilon}(u)|^{2} +  
\frac{3\varepsilon^{4}}{8}\int_{\mathbb R^{3}} |\nabla \phi_{\varepsilon}(u)|^4\in \mathbb R
\end{equation*}
in such a way that we can write
$$J_{\lambda,\varepsilon}(u) = \frac{1}{2}\|u\|_{H^{1}}^{2} + I_{\varepsilon}(u) -\lambda\int_{\mathbb R^{3}} F(x,u)
-\frac{1}{2^{*}}\int_{\mathbb R^{3}}|u|^{2^{*}}.$$
With this notation it is
$$I_{0} (u)= \frac{1}{4}\int_{\mathbb R^{3}}|\nabla \phi_{0}(u)|^{2},$$
where of course $-\Delta\phi_{0}(u) = u^{2}$ in $\mathbb R^{3}$.

\begin{remark}
We observe that in \cite{BK} it has been proved by hands that $J_{\lambda,\varepsilon}$
is $C^{1}$ and that its critical points are solutions of \eqref{eq:equation}.
Of course the nontrivial part there was to show that $I_{\varepsilon}$ is $C^{1}$
with Frech\'et derivative $I_{\varepsilon}'(u)$ given by
$$\forall v\in H^{1}(\mathbb R^{3}): I_{\varepsilon}'(u)[v] =\int_{\mathbb R^{3}} \phi_{\varepsilon}(u) u v.$$
In other words,  it gives rise exactly to the nonlocal term $\phi_{\varepsilon}(u)u$ in the equation \eqref{eq:equation}.
See \cite[Proposition 4.1]{BK}.
\end{remark}


\begin{remark}\label{rem:derivada}
A useful consequence of the differentiability of $I_{\varepsilon}$ 
is that, if $v\in H^{1}(\mathbb R^{3})$ is fixed,
then the function $t\in(0,\infty)\mapsto I_{\varepsilon}(tv)$ is $C^{1}$ with
\begin{equation*}
\frac{d}{dt} I_{\varepsilon}(tv) = I_{\varepsilon}'(tv)[v]= t\int_{\mathbb R^{3}} \phi_{\varepsilon}(tv) v^{2}.
\end{equation*}
\end{remark}

\medskip

In view of the above arguments, we are reduced to find a solution $u_{\lambda,\varepsilon}$
of equation \eqref{eq:equation}, that is a critical point of the functional $J_{\lambda,\varepsilon}$.

\section{The truncated functional}\label{sec:trunc}

In order to overcome the lack of compactness and the ``growth'' of order $4$ in $I_{\varepsilon}$, let us define a truncation for the functional $J_{\lambda,\varepsilon}$ in the following way. Consider a smooth cut-off function $\psi:[0,+\infty)\to \mathbb R_{+}$ such that
$$
\left\{\begin{array}{lll} 
\psi(t) = 1, & t\in[0,1],\smallskip\\
0 \leq \psi(t) \leq 1, & t\in(1,2),\smallskip\\
\psi(t) = 0, &t\in[2,\infty),\smallskip\\
|\psi'|_{\infty}\leq 2.
\end{array}\right.
$$
For each $T>0$ we define  $h_T(u) := \psi\left({\|u\|_{H^{1}}^2}/{T^2}\right)$ and the  truncated functional 
$J_{\lambda,\varepsilon}^{T}:H^{1}(\mathbb{R}^{3})\rightarrow\mathbb{R}$ given by 
\begin{eqnarray*}
J_{\lambda,\varepsilon}^{T}(u) &:= &\frac 12 \|u\|_{H^{1}}^{2}+ 
h_{T}(u)\biggl[\frac 14 \int_{\mathbb{R}^{3}}|\nabla \phi_{\varepsilon}(u)|^{2}  
+ \frac{3\varepsilon^{4}}{8} \int_{\mathbb{R}^{3}}|\nabla \phi_{\varepsilon}(u)|^{4} \biggl] 
- \lambda\displaystyle\int_{\mathbb{R}^{3}}F(x,u) - \frac{1}{2^{*}}\displaystyle\int_{\mathbb{R}^{3}}|u|^{2^{*}}\\
&=&\frac 12 \|u\|_{H^{1}}^{2}+  h_{T}(u)I_{\varepsilon}(u)- \lambda\displaystyle\int_{\mathbb{R}^{3}}F(x,u) 
- \frac{1}{2^{*}}\displaystyle\int_{\mathbb{R}^{3}}|u|^{2^{*}}.
\end{eqnarray*}
The functional $J_{\lambda,\varepsilon}^{T}$ is $C^{1}$ with differential given, for all $u,v \in H^{1}(\mathbb{R}^{3})$, by
\begin{multline}\label{eq:derivada}
(J_{\lambda,\varepsilon}^{T})'(u)[v]=\langle u, v\rangle_{H^{1}}
+\frac{2}{T^{2}}\psi'\left(\frac{\|u\|_{H^{1}}^{2}}{T^{2}}\right)\langle u, v\rangle_{H^{1}} I_{\varepsilon}(u)\\
 +h_{T}(u)\int_{\mathbb R^{3}} \phi_{\varepsilon}(u) u v 
-\lambda\int_{\mathbb R^{3}} f(x,u)v  - \int |u|^{2^{*}-1}v.
\end{multline}
Then  $u_{\lambda,\varepsilon}\in H^{1}(\mathbb R^{3})$ is a critical point of $J_{\lambda,\varepsilon}^{T}$,
if and only if  the pair 
$(u_{\lambda,\varepsilon},\phi_{\varepsilon}(u_{\lambda,\varepsilon}))\in H^{1}(\mathbb R^{3})\times X$ is a weak solution of
\begin{equation*}
\left\{
\begin{array}[c]{ll}
(- \Delta u +u)\left(1+\dis\frac{2}{T^{2}}\psi'\left(\frac{\|u\|_{H^{1}}^{2}}{T^{2}}\right)I_{\varepsilon}(u) \right)+h_{T}(u)\phi u = \lambda f(x,u)+|u|^{2^{*}-2}u & \ \mbox{in} \ \
\mathbb{R}^{3}, \ \ \\ -\Delta \phi - \varepsilon^{4}\Delta_4 \phi = u^{2}& \  \mbox{in} \ \ \mathbb{R}^{3}.
\end{array}
 \right.
\end{equation*}
Let us observe the following
\begin{lemma}\label{lem:proiect}
Let $T,\lambda, \varepsilon>0$ be fixed.
For every $v\in H^{1}(\mathbb R^{3})\setminus\{0\}$, the function
$$t\in[0,+\infty) \mapsto J_{\lambda,\varepsilon}^{T}(tv) \in\mathbb R$$
has a global maximum  point  which does not depend on $\varepsilon$ and is 
 strictly positive. It will be denoted hereafter with $t_{\lambda}^{T}(v)$.
Moreover, 
$$\forall T>0: \ 
 \lim_{\lambda\to+\infty} {t_{\lambda}^{T}(v)}=0.$$
\end{lemma}
\begin{proof}
First of all let us see the existence of such 
$t_{\lambda}^{T}(v)$ for every $T,\lambda,\varepsilon>0$. 

It follows from  \eqref{f_{1}} and \eqref{f_{2}}
that, for each $\eta>0$, there exists a positive constant
$C(\eta)$ such that
\begin{eqnarray}\label{ref2}
F(x,t)\leq \eta \frac{1}{2}|t|^{2} +
\frac{1}{q}C(\eta)|t|^{q}.
\end{eqnarray}
Then, fixed $v\neq0$,
\begin{eqnarray*}\label{eq:}
J_{\lambda,\varepsilon}^{T}(tv)&=& \frac{t^{2}}{2} \|v\|_{H^{1}}^{2}+ 
h_{T}(tv) I_{\varepsilon}(tv)
- \lambda\displaystyle\int_{\mathbb{R}^{3}}F(x,tv) 
- \frac{t^{2^{*}}}{2^{*}}\displaystyle\int_{\mathbb{R}^{3}}|v|^{2^{*}} \\
&\geq&\frac{t^{2}}{2} \|v\|_{H^{1}}^{2}- \lambda\displaystyle\int_{\mathbb{R}^{3}}F(x,tv) 
- \frac{t^{2^{*}}}{2^{*}}\displaystyle\int_{\mathbb{R}^{3}}|v|^{2^{*}} 
\end{eqnarray*}
and choosing $\eta$ sufficiently small and using the Sobolev embeddings, we have
$$J_{\lambda,\varepsilon}^{T}(tv)\geq \frac{t^{2}}{2}(1-C_{1}\eta \lambda) \|v\|_{H^{1}}^{2} - t^{q}C_{2}C(\eta)\lambda\|v\|_{H^{1}}^{q}- t^{2^{*}} C_{3}\|v\|_{H^{1}}^{2^{*}}>0 \quad \text{for small } t.$$
Here  $C_{i}, i=1,2,3$ are  the embedding constant of $H^{1}(\mathbb R^{3})$, respectively,
 into $L^{2}(\mathbb  R^{3}),L^{q}(\mathbb  R^{3})$ and $L^{2^{*}}(\mathbb  R^{3}).$ 
In the previous inequalities the dependence on $\varepsilon$ disappeared since
the term involving $\varepsilon$ was thrown away being positive.

On the other hand it is easily seen that 
$$\lim_{t\to+\infty}J_{\lambda,\varepsilon}^{T}(tv)= \lim_{t\to+\infty}
\left(\frac{t^{2}}{2} \|v\|_{H^{1}}^{2}+ 
h_{T}(tv) I_{\varepsilon}(tv)
- \lambda\displaystyle\int_{\mathbb{R}^{3}}F(x,tv) 
- \frac{t^{2^{*}}}{2^{*}}\displaystyle\int_{\mathbb{R}^{3}}|v|^{2^{*}} \right)
=-\infty$$
and again we  avoid the dependence on $\varepsilon$ since for $s$ large $h_{T}(s)=0$.
Then the existence of $t_{\lambda}^{T}(v)>0$, which does not depend on $\varepsilon$, is guaranteed.

Now let us prove the limit.
We set for brevity   $t_{\lambda}:=t_{\lambda}^{T}(v)$. Such a $t_{\lambda}$ satisfies (recall Remark \ref{rem:derivada}):
\begin{equation}\label{eq:converg}
t_{\lambda}^{2} +\frac{2t_{\lambda}^{2}}{T^{2}}\psi'\biggl(\frac{t_{\lambda}^{2} }{T^{2}}\biggl) I_{\varepsilon}(t_{\lambda}v) +
t_{\lambda}^{2} \psi\left(\frac{t_{\lambda}^{2}}{T^{2}}\right)\int_{\mathbb R^{3}} \phi_{\varepsilon}(t_{\lambda} v)v^{2}   
 =\\
\lambda \int_{\mathbb{R}^{3}} f(x,t_{\lambda} v)t_{\lambda }v +t_{\lambda}^{2^{*}}
\int_{\mathbb{R}^{3}}|v|^{2^{*}}
\end{equation}
and then, by  \eqref{f_{3}},
$$t_{\lambda}^{2}+t_{\lambda}^{2} \psi\left(\frac{t_{\lambda}^{2}}{T^{2}}\right)\int_{\mathbb R^{3}} \phi_{\varepsilon}
(t_{\lambda} v)v^{2}   
 \geq t_{\lambda}^{2^{*}}\int_{\mathbb{R}^{3}}|v|^{2^{*}}.
$$
If it were $\lim_{\lambda\to +\infty}t_\lambda = +\infty$,  then, there exists $\widetilde\lambda>0$
such that for every $\lambda\geq\widetilde\lambda$ it is $t_{\lambda}>\sqrt2 T$ and then the previous
inequality becomes
$$t_{\lambda}^{2} \geq t_{\lambda}^{2^{*}}\int_{\mathbb{R}^{3}}|v|^{2^{*}}.
$$
But this is impossible for $\lambda$ large.
 Thus, $\lim_{\lambda\to+\infty}t_{\lambda}=\beta\ge0$.
Of course we need to show that $\beta=0.$

If $\beta>0$,
 coming back to \eqref{eq:converg},  and recalling that by Lemma \ref{lem:facile}
 it is $\phi_{\varepsilon}(t_{\lambda} v) \to \phi_{\varepsilon}(\beta v)$ in $L^{\infty}(\mathbb R^{3})$
 as $\lambda\to+\infty$,
 we deduce
\begin{eqnarray*}
+\infty \longleftarrow  \lambda\dis\int_{\mathbb R^{3}}f(x,t_{\lambda}v)t_{\lambda}v
+t_{\lambda}^{2^{*}}\dis\int_{\mathbb R^{3}}|v|^{2^{*}}   \leq \beta^{2}+\beta^{2}|\psi|_{\infty} M
+o_{n}(1) 
\quad \text{ as } \lambda\to +\infty.
\end{eqnarray*}
which is an absurd. Thus we conclude that $\beta=0$. This means that for every fixed
$T>0$ it is $\lim_{\lambda\to+\infty}t^{T}_{\lambda}(v)=0$, completing the proof.
\end{proof}

%
%
%

\subsection{The Mountain Pass Geometry for $J_{\lambda,\varepsilon}^{T}$}

In the sequel, we prove that the functional $
J_{\lambda,\varepsilon}^{T}$ has the
Mountain Pass Geometry
with some kind of uniformity with respect to the parameters.
Observe that even  $\varepsilon=0$ is allowed,
up to change $\phi_{\varepsilon}$ into $\phi_{0}$, the solution of the Poisson equation,
and the function space $X$ into $D^{1,2}(\mathbb R^{3})$.

\begin{lemma}\label{geometria1}
Assume that conditions  \eqref{f_{1}} and \eqref{f_{2}} hold.
Then, for every $\lambda>0$ there exists
 numbers $\rho_{\lambda},\alpha_{\lambda}>0 $ such that,
$$
\forall T>0,\ \forall \varepsilon\ge0: \quad 
J_{\lambda,\varepsilon}^{T}(u)\geq \alpha_{\lambda},\quad  \text{whenever }\
\|u\|_{H^{1}}=\rho_{\lambda}.
$$
\end{lemma}
We observe explicitly that indeed  $\rho_{\lambda}, \alpha_{\lambda}$
does not depend on $T$ neither on $\varepsilon$;
indeed in the proof we will simply through away the term  $h_{T}(u)I_{\varepsilon}(u)$,
being positive.
\begin{proof}
Let $\lambda>0$ be fixed.
It follows from  \eqref{f_{1}} and \eqref{f_{2}}
that, for each $\eta>0$, there exists a positive constant
$C(\eta)$ such that
\begin{eqnarray}\label{ref2}
F(x,t)\leq \eta \frac{1}{2}|t|^{2} +
\frac{1}{q}C(\eta)|t|^{q}.
\end{eqnarray}

By (\ref{ref2}) we have
$$
J_{\lambda,\varepsilon}^{T}(u) \geq  \frac{1}{2}\|u\|_{H^{1}}^{2} -
\frac{\eta}{2}\int_{\mathbb{R}^{3}}|u|^{2}
-\frac{1}{q}C(\eta)\lambda\dis\int_{\mathbb{R}^{3}}|u|^{q}-
\dis\frac{1}{2^{*}}\dis\int_{\mathbb{R}^{3}}|u|^{2^{*}}.
$$
So, using the Sobolev Embedding Theorem, there is a positive constant $C>0$ such that
$$
J_{\lambda,\varepsilon}^{T}(u) \geq  C\|u\|_{H^{1}}^{2} -\lambda C\|u\|_{H^{1}}^{q} -
C\|u\|_{H^{1}}^{2^{*}}.
$$
Since $2< q< 2^{*}$, the result follows by choosing $\rho_{\lambda}>0$ small
enough.\end{proof}

\begin{lemma}\label{segundageometria1}
Assume that conditions \eqref{f_{1}}-\eqref{f_{3}} hold. 
Then for every $T>0$, there exists $e_{T} \in
H^{1}(\mathbb{R}^{N})$ such that
$$\forall \lambda>0,\ \forall \varepsilon\ge0: \quad J_{\lambda,\varepsilon}^{T}(e_{T})<0\quad \text{and } \quad \| e_{T}\|_{H^{1}} >\rho_{\lambda},$$
where $\rho_{\lambda}$ is given in Lemma \ref{geometria1}.
\end{lemma}
Here the fact that $e_{T}$ does not depends on $\varepsilon$ is a consequence of the 
fact that the with the truncation we kill the term $I_{\varepsilon}$.
\begin{proof}
Let $T>0$ be fixed. Let now $v\in C^{\infty}_{0}(\mathbb{R}^{N})$, positive, with $\|v\|_{H^{1}}=1$.
Using \eqref{f_{3}} and considering $t>2T$,  we get
$$
J_{\lambda,\varepsilon}^{T}(tv)\leq \frac 12
t^{2}-\lambda t^{\theta}\int_{\mathbb R^{3}} v^{\theta} -
\frac{t^{2^{*}}}{2^{*}}\int_{\mathbb R^{3}} v^{2^{*}}< 
\frac 12
t^{2} - \frac{t^{2^{*}}}{2^{*}}\int_{\mathbb R^{3}} v^{2^{*}}
$$
Since $2< \theta$,  the result follows by choosing some 
$t_{*}>2T$ large enough and setting
$e_{T}:=t_{*}v$.
\end{proof}

We recall that a sequence $\{u_{n}\}\subset H^{1}(\mathbb{R}^{3})$ is a
Palais-Smale sequence for the functional $J_{\lambda,\varepsilon}^{T}$ at the
level $d \in \mathbb{R}$ if
$$
J_{\lambda,\varepsilon}^{T}(u_{n})\rightarrow d \quad \mbox{and } \quad
(J_{\lambda,\varepsilon}^{T})'(u_{n})\rightarrow 0 \ \mbox{in} \ H^{-1}(\mathbb{R}^{3}).
$$
If every Palais-Smale sequence of $
J_{\lambda,\varepsilon}^{T}$ has a strong
convergent subsequence, then one says that $
J_{\lambda,\varepsilon}^{T}$ satisfies
the Palais-Smale condition, or $(PS)$ for short.

Then, since 
 the functional $J_{\lambda,\varepsilon}^{T}$ satisfies the geometric assumptions of 
Mountain Pass Theorem  (see \cite{Ambrosetti}), we know that
(see  \cite[p.12]{Willem}), for every $T,\lambda>0,\varepsilon\geq0$ there exists a sequence $\{u_{n}\}\subset
H^{1}(\mathbb{R}^{3})$  satisfying
$$
J_{\lambda,\varepsilon}^{T}(u_{n})\rightarrow c^{T}_{\lambda,\varepsilon}>0 \ \ \mbox{and} \ \
(J_{\lambda,\varepsilon}^{T})'(u_{n})\rightarrow 0,
$$
where
$$
c^{T}_{\lambda,\varepsilon} := \dis\inf_{\gamma \in \Gamma^{T}_{\lambda,\varepsilon}} \dis\max_{t \in [0,1]}
J_{\lambda,\varepsilon}^{T}(\gamma(t))>0
$$
and
$$
\Gamma^{T}_{\lambda,\varepsilon} := \{ \gamma \in C([0,1],H^{1}(\mathbb R^{3})) : \gamma(0)=0,
~J_{\lambda,\varepsilon}^{T}(\gamma(1)) < 0\}.
$$
It is clear that this sequence should depend also on $T,\lambda,\varepsilon$
but we omit this for simplicity.
In other words, $\{u_{n}\}$ is a $(PS)$ sequence at level $c_{\lambda,\varepsilon}^{T}$
for the functional $J_{\lambda,\varepsilon}^{T}$.

Observe that, since $e_{T}$ found in Lemma \ref{segundageometria1}
 does not depends on  $\lambda$ neighter on $\varepsilon$,  by defining the path
\begin{equation}\label{eq:gammastar}
\gamma_{*}: t\in[0,1]\mapsto t e_{T}\in H^{1}(\mathbb R^{3})
\end{equation}
we get 
$\gamma_{*}\in\bigcap_{\lambda>0,\varepsilon\geq0}\Gamma_{\lambda,\varepsilon}^{T}$.
%

\begin{remark}\label{rem:epsilon}
Observe that the Mountain Pass structure
of $J_{\lambda,\varepsilon}^{T}$ does not depend on $\varepsilon\geq0$.
\end{remark}

\subsection{Estimates of $c_{\lambda,\varepsilon}^{T}$}
Here we study the behaviour of  the Mountain Pass levels 
$c_{\lambda,\varepsilon}^{T}$ with respect to $\lambda$,
whenever $T,\varepsilon$ are fixed.
This will be fundamental in order to prove that the $(PS)$
sequences at level $c_{\lambda,\varepsilon}^{T}$ are bounded
for $\lambda$ large.

Actually the next result is again independent on $\varepsilon\geq0$,
which is not surprising in view of Remark \ref{rem:epsilon}.
\begin{lemma}\label{nivel}
If the conditions \eqref{f_{0}}-\eqref{f_{3}} hold, then 
$$
\forall T>0 : \  \lim_{\lambda\rightarrow +\infty}\sup_{\varepsilon\geq0}c^{T}_{\lambda,\varepsilon}=0.
$$
\end{lemma}
\begin{proof}
Let $T>0$ be fixed. We prove that 
for every $\eta>0$ there exists $\widetilde\lambda>0$ such that
$$\forall \lambda>\widetilde\lambda:
\quad 0<\max_{t\in[0,1]} J_{\lambda,\varepsilon}^{T}(\gamma_{*}(t))<\eta,
\quad  \forall \varepsilon\geq0.
$$
This of course will give the conclusion.

Then let us fix $\eta>0$.
 Let $v\in C^{\infty}_{c}(\mathbb R^{N}), v\geq0$ with $\|v\|=1$ be the same
function fixed in the proof of Lemma \ref{segundageometria1}.
By Lemma \ref{lem:proiect} there exists
$t^{T}_{\lambda}=t^{T}_{\lambda}(v)>0$ verifying
$J_{\lambda,\varepsilon}^{T}(t^{T}_{\lambda}v)=\max_{t\geq
0}J_{\lambda,\varepsilon}^{T}(tv)$
and  $\lim_{\lambda\to+\infty}\sup_{\varepsilon\geq0}t^{T}_{\lambda}=0.$

Hence, due to the continuity of the maps $h_{T}$ and $I_{\varepsilon}$ we get
the uniform limits in $\varepsilon\geq0$:
$$
\lim_{\lambda\rightarrow+
\infty}h_{T}(t^{T}_{\lambda}v)=1, \qquad \lim_{\lambda\to+\infty} I_{\varepsilon}(t^{T}_{\lambda} v)=0.
$$
Then there exists $\widetilde\lambda>0$ such that $$\forall \lambda>\widetilde\lambda: 
\quad \frac{1}{2}(t^{T}_{\lambda})^{2}+h_{T}(t^{T}_{\lambda}v)I_{\varepsilon}(t^{T}_{\lambda} v)< \eta, \quad \forall \varepsilon\geq0.$$
Since, as we know,   for $\lambda>\widetilde\lambda$ it is (see \eqref{eq:gammastar})
 $\gamma_{*} \in  \cap_{\varepsilon\geq0}\Gamma_{\lambda,\varepsilon}^{T}$,
 we get the following estimate:
\begin{eqnarray*}\label{eq:clambda}
0<
 \max_{t \in [0,1]}J_{\lambda,\varepsilon}^{T}(\gamma_{*}(t))
 &=&J_{\lambda,\varepsilon}^{T}(t^{T}_{\lambda}v)  \\ 
&\leq&
\dis\frac{1}{2}(t^{T}_{\lambda})^{2} + h_{T}(t^{T}_{\lambda}v)\biggl[\frac 14 \displaystyle\int_{\mathbb{R}^{3}}|\nabla \phi_{\varepsilon}(t^{T}_{\lambda}v)|^{2}  + \frac{ 3\varepsilon^{4}}{8} \int_{\mathbb{R}^{3}}|\nabla \phi_{\varepsilon}(t^{T}_{\lambda}v)|^{4} \biggl] \\
&=& \dis\frac{1}{2}(t^{T}_{\lambda})^{2} + h_{T}(t^{T}_{\lambda}v)I_{\varepsilon}(t^{T}_{\lambda}v)\\
&<&\eta
\end{eqnarray*}
concluding the proof.
\end{proof}
We remark explicitly  the important fact that the limit in Lemma \ref{nivel}
is uniform in $\varepsilon\geq0$.

Thanks to the previous Lemma we have the following important result.
Recall that $\theta\in (4,2^{*})$ is the constant given in the Ambrosetti-Rabinowitz condition
\eqref{f_{3}}.
\begin{lemma}\label{limitacao}
Let $T>0$ be fixed and
let $\lambda$ sufficiently large, let us say  $\lambda\geq\lambda(T),$  such that 
$$\sup_{\varepsilon\geq0} c^{T}_{\lambda,\varepsilon}< \frac{\theta-2}{2\theta}T^{2},
$$
(this is possible, in view of the previous Lemma \ref{nivel}). 

Then, given $\varepsilon\geq0$, any $(PS)$ sequence $\{u_{n}\}$  (we do not write the dependence on $T,\lambda,\varepsilon$)
at level  $c^{T}_{\lambda,\varepsilon}$ for the functional $J_{\lambda,\varepsilon}^{T}$ is bounded;
more precisely it satisfies, up to subsequence,
$\|u_{n}\|_{H^{1}}\leq T$.
In particular, for all $\lambda\geq\lambda(T), \varepsilon\geq0$,
the sequence $\{u_{n}\}$ is such that
$$J_{\lambda,\varepsilon}(u_{n}) \to c_{\lambda,\varepsilon},\quad J_{\lambda,\varepsilon}'(u_{n})\to0,$$
that is, it is a $(PS)$ sequence
at the Mountain Pass level $c_{\lambda,\varepsilon}$ for the untruncated functional $J_{\lambda,\varepsilon}$.
\end{lemma}
\begin{proof}
Given $\varepsilon\geq0$,
let us first show that $\{u_{n}\}$ is bounded by $2T^{2}.$
Assume by contradiction that there exists a subsequence of $\{u_{n}\}$, still denoted with $\{u_{n}\}$,
such that $\|u_{n}\|_{H^{1}}^{2}> 2T^{2}$. Taking into account  \eqref{eq:derivada}
and \eqref{f_{3}}, it follows that
\begin{eqnarray*}
c^{T}_{\lambda,\varepsilon} &=&J_{\lambda,\varepsilon}^{T}(u_{n})-
\dis\frac{1}{\theta}(J_{\lambda,\varepsilon}^{T})'(u_{n})[u_{n}] + o_{n}(1)
\\
&\geq&\frac{\theta-2}{\theta}\|u_{n}\|_{H^{1}}^{2} +\psi\left( \frac{\|u_{n}\|_{H^{1}}^{2}}{T^{2}}\right) 
\Big[I_{\varepsilon}(u_{n})-\frac{1}{\theta}\int_{\mathbb R^{3}} \phi_{\varepsilon}(u_{n})u_{n}^{2}\Big] \\
&& -  \frac{2}{\theta T^{2}}\psi'\left(\frac{\|u_{n}\|_{H^{1}}^{2}}{T^{2}}\right)\|u_{n}\|_{H^{1}}^{2}I_{\varepsilon}(u_{n})+o_{n}(1)\\
&\geq& 
\frac{\theta-2}{\theta} T^{2}+ o_{n}(1)
\end{eqnarray*}
being $\psi'\leq0$. This is a contradiction  and proves that $\|u_{n}\|_{H^{1}}^{2}\leq 2T^{2}$. We can prove now the Lemma.
Assume by contradiction that $T^{2}<\|u_{n}\|_{H^{1}}^{2}\leq 2T^{2}$. We have, with similar computations as before
and using that $\psi$ is decreasing, that
\begin{eqnarray*}\label{eq:}
c_{\lambda,\varepsilon}^{T} &=& J_{\lambda,\varepsilon}^{T}(u_{n}) - \frac{1}{\theta}(J_{\lambda,\varepsilon}^{T})'(u_{n})[u_{n}]+o_{n}(1)\\
&\geq& \frac{\theta-2}{2\theta}T^{2}+\psi(2)\Big[I_{\varepsilon}(u_{n})-\frac{1}{\theta}\int_{\mathbb R^{3}} \phi_{\varepsilon}(u_{n})u_{n}^{2}\Big] +o_{n}(1)\\
&=& \frac{\theta-2}{2\theta}T^{2} +o_{n}(1)
\end{eqnarray*}
which contrasts with  the assumption and conclude the proof.
\end{proof}

\section{Proof of Theorem \ref{teorema1}}\label{sec:final}

Here we prove Theorem \ref{teorema1} hence $T$ and $\varepsilon$ have to be considered
fixed. 
From Lemma \ref{nivel}  there exists $\lambda'(T)>0$
(actually which  does not depend  on $\varepsilon$ being the limit in Lemma \ref{nivel}
uniform in $\varepsilon$)
such that
\begin{eqnarray}\label{ref4}
\forall \lambda\geq \lambda'(T) : \  c^{T}_{\lambda,\varepsilon}<\frac{2^{*}-\theta}{2^{*}\theta}S^{3/2}\qquad  \forall \varepsilon\geq0.
\end{eqnarray}
Here, $S$ is the best constant
for the embedding $H^1(\mathbb{R}^{3})\hookrightarrow
L^{2^{*}}(\mathbb{R}^{3})$. 

In the remaining of this Section, the fact that inequality \eqref{ref4}
is independent on $\varepsilon$ will not be used. However it will be important in the final
Section \ref{sec:T3}.

Now, fix $\lambda\geq\max\{\lambda'(T),\lambda(T)\}$ where
$\lambda(T)$ is given in Lemma \ref{limitacao}.
Let us 
show that the  truncated functional $J_{\lambda,\varepsilon}^{T}$ 
admits a critical point with norm less then $T$; then
this will  be a critical point of $J_{\lambda,\varepsilon}$
and hence
a solution of our problem \eqref{eq:P}.

From
Lemmas \ref{geometria1}, \ref{segundageometria1} and \ref{limitacao} there exists a
bounded $(PS)$ sequence $\{u_{n}\} \subset H^{1}(\mathbb{R}^{N})$ 
at level $c^{T}_{\lambda,\varepsilon}$  for the functional $J_{\lambda,\varepsilon}^{T}$.
Since the sequence $\{u_{n}\}$ verifies also $\|u_{n}\|_{H^{1}}\leq T$, then 
it is actually a $(PS)$ sequence for the functional $J_{\lambda,\varepsilon}$ at level
$c_{\lambda,\varepsilon} = c_{\lambda,\varepsilon}^{T}$
and we can assume that 
there exists $u_{\lambda,\varepsilon}^{T}\in H^{1}(\mathbb R^{3})$ such that 
$u_{n}\rightharpoonup u^{T}_{\lambda,\varepsilon}$ in $H^{1}(\mathbb R^{3})$
and $\|u_{\lambda,\varepsilon}^{T}\|_{H^{1}}\leq T$.

\medskip

We show now the following 

\medskip

{\bf Claim: }  $\|u_{n}\|_{H^{1}} \to \|u^{T}_{\lambda,\varepsilon}\|_{H^{1}}$ as $n \rightarrow \infty$.

%
%


\bigskip

It order to prove the Claim we
 suppose, up to a subsequence,  that
\begin{equation}\label{eq:Lions}
|\nabla u_n|^2 \rightharpoonup |\nabla u_{\lambda,\varepsilon}^{T}|^2 + \mu~~\text{ and }
\quad|u_n|^{2^{*}} \rightharpoonup |u_{\lambda,\varepsilon}^{T}|^{2^{*}} +
\nu\quad\text{(weak$^*$-sense of measures).}
\end{equation}
Using the Concentration Compactness Principle due to Lions (see
\cite[Lemma 2.1]{Lio2}), we get the existence of a set, at most countable
$\Lambda$, sequences $\{x_i\}_{i\in \Lambda} \subset \mathbb{R}^3$, $\{\mu_i\}_{i\in \Lambda}, \{\nu_i\}_{i\in \Lambda}
\subset [0,\infty)$, such that
\begin{equation}
\nu  =  \sum_{i \in \Lambda}\nu_{i}\delta_{x_{i}},~~\mu\geq \sum_{i
\in \Lambda}\mu_{i}\delta_{x_{i}}~~\text{ and }~~S
\nu_{i}^{2/2^{*}}\leq \mu_{i} \ \  \forall i \in\Lambda,
 \label{lema_infinito_eq11}
\end{equation}
 where $\delta_{x_i}$ is the Dirac mass centered in 
$x_i \in \mathbb R^{3}$.

\medskip

Note that, if it were $\nu_i \geq S^{3/2}$ for
some $i \in \Lambda$, since $\{u_n\}$ is a $(PS)$ sequence for $J_{\lambda,\varepsilon}$
at level $c_{\lambda,\varepsilon}$, we have
\begin{eqnarray*}
c_{\lambda,\varepsilon} &=& J_{\lambda,\varepsilon}(u_n) - \displaystyle\frac{1}{\theta}
J_{\lambda,\varepsilon}'(u_n)[u_n] + o_n(1)\\
&=&\frac{\theta-2}{2\theta}\|u\|_{H^{1}}^{2} 
+\frac{\theta-4}{4\theta} \int_{\mathbb R^{3}}|\nabla \phi_{\varepsilon}(u_{n})|^{2}
+\frac{8-3\theta}{8\theta}\int_{\mathbb R^{3}}|\nabla \phi_{\varepsilon}(u_{n})|^{4}\\
&\ &+\lambda\int_{\mathbb R^{3}}\left(\frac{1}{\theta}f(x,u_{n})u_{n}-F(x,u_{n}) \right)+ \frac{2^{*}-\theta}{2^{*}\theta}
\int_{\mathbb R^{3}}|u_{n}|^{2^{*}} \\
&\geq&\frac{2^{*}-\theta}{2^{*}\theta} \int_{\mathbb R^{3}} |u_{n}|^{2^{*}}\psi_{\rho} 
\end{eqnarray*}
Then, passing to the limit in $n$, 
$$c_{\lambda,\varepsilon} \geq \frac{2^{*}-\theta}{2^{*}\theta}\left( \int_{\mathbb R^{3}} |u|^{2^{*}}+\int_{\mathbb R^{3}} \sum_{i\in \Lambda} \delta_{x_{i}} \psi_{r}\right)
\geq \frac{2^{*}-\theta}{2^{*}\theta} \nu_{i}\geq \frac{2^{*}-\theta}{2^{*}\theta} S^{3/2}$$
%
which is absurd for our choice of  $\lambda$.
 Thus  it has necessarily to be 
 \begin{equation}\label{eq:menor}
 \forall i\in \Lambda: \  \nu_{i}<S^{3/2}.
 \end{equation}

\medskip


On the other hand, fix $i \in \Lambda$. Consider
$\psi \in C_0^{\infty}(\mathbb R^{3},[0,1])$ such that $\psi \equiv 1$ on
$B_1(0)$, $\psi \equiv 0$ on $\mathbb R^{3} \setminus B_2(0)$ and $|\nabla
\psi|_{\infty} \leq 2$. Defining $\psi_{r}(x) :=
\psi((x-x_i)/r)$ where  $r>0$, we have that
$\{\psi_{r}u_n\}$ is in $H^{1}(\mathbb R^{3})$. Since  $\|u_n\|\leq T$, it holds 
$J_{\lambda}'(u_n)[\psi_{r}u_n] \to 0$, explicitely,
\begin{eqnarray}\label{eq:nrho}
\int_{\mathbb R^{3}} u_{n}\nabla u_{n}\nabla \psi_{r}+\int_{\mathbb R^{3}}|\nabla u_{n}|^{2}\psi_{r}
+\int_{\mathbb R^{3}} \phi_{u_{n}} u_{n}^{2} \psi_{r}
 -\lambda\int_{\mathbb R^{3}} f(x,u_{n}) u_{n} \psi_{r} -\int_{\mathbb R^{3}}|u_{n}|^{2^{*}}\psi_{r}=o_{n}(1)
 \end{eqnarray}
Let us pass to the limit, first as $n\to \infty$ and then as $r\to0$, in \eqref{eq:nrho}. We first note that
$$
\Big|\int_{\mathbb{R}^{3}}u_{n}\nabla u_{n}  \nabla
\psi_{r} \Big|\leq \int_{B_{2r}(x_{i})} |\nabla
u_{n}||u_{n} \nabla \psi_{r}|\leq C \left( \int_{B_{2r}(x_{i})}|u_{n}
\nabla \psi_{r}|^{2} \right)^{1/2},
$$
and then
\begin{equation*}\label{eq:}
\limsup_{n\to+\infty} \big| \int_{\mathbb{R}^{3}}u_{n}\nabla u_{n}  \nabla
\psi_{\varrho} \big|\leq C \left( \int_{ B_{2r}(x_{i})} |u \nabla \psi_{r}|^{2} \right)^{1/2} 
\end{equation*}
from which
\begin{equation}\label{eq:serve1}
\lim_{r\to 0}\left( \limsup_{n\to+\infty}  \Big|\int_{\mathbb{R}^{3}}u_{n}\nabla u_{n}  \nabla
\psi_{\varrho}\Big| \right)=0.
\end{equation}
Analogously it is easy to see that
\begin{equation}\label{eq:serve2}
\lim_{r\to 0} \left( \limsup_{n\to +\infty}\int_{\mathbb R^{3}}\phi_{\varepsilon}(u_{n}) u^{2}_{n}\psi_{r} \right) = 
\lim_{r\to 0} \left( \limsup_{n\to +\infty}\int_{\mathbb R^{3}}f(x,u_{n})u_{n}\psi_{r} \right)= 0
\end{equation}
Moreover we have
\begin{eqnarray}\label{eq:}
&\displaystyle\lim_{n\to+\infty}\int_{\mathbb R^{3}}|\nabla u_{n}|^{2}\psi_{r} \geq \int_{\mathbb R^{3}}\psi_{r}d\mu &\\
&\displaystyle\lim_{n\to+\infty} \int_{\mathbb R^{3}} |u_{n}|^{2^{*}}\psi_{r} \geq \int_{\mathbb R^{3}} \psi_{r}d\nu&
\end{eqnarray}
in \eqref{eq:nrho}, taking into account \eqref{eq:Lions}, \eqref{eq:serve1} and \eqref{eq:serve2},
by \eqref{eq:nrho} we deduce
$$
\int_{\mathbb{R}^{3}} \psi_{r}\textrm{d}\nu \geq \int_{\mathbb{R}^{3}}
 \psi_{r}\textrm{d}\mu +o_{r}(1).
$$
But then passing to the limit as $r\to0$ we get $\nu_{i}\geq\mu_{i}$ and
by \eqref{lema_infinito_eq11}, we infer  that
\begin{equation*}
\nu_i \geq S^{3/2}.
\end{equation*}
This of course contrasts with  \eqref{eq:menor} and gives that  $\Lambda=\emptyset$.

As a consequence of this,
 $u_n \to u_{\lambda,\varepsilon}^{T}$ in $L^{2^{*}}(\mathbb{R}^{3})$, from which we deduce in a standard
 way that  that $\|u_{n}\|_{H^{1}} \to \|u_{\lambda,\varepsilon}^{T}\|_{H^{1}}$, proving the Claim.

Then $u_{n} \to u^{T}_{\lambda,\varepsilon}$ in $H^1(\mathbb{R}^{3})$ and hence 
 since the functional $J_{\lambda,\varepsilon}$ is $C^{1}$ and $\|u_{n}\|_{H^{1}}\leq T$:
$$
J_{\lambda,\varepsilon}(u_{n})=J_{\lambda,\varepsilon}^{T}(u_{n})\rightarrow J_{\lambda,\varepsilon}^{T}(u^{T}_{\lambda,\varepsilon})=c_{\lambda,\varepsilon}^{T}=c_{\lambda,\varepsilon} \ \ \mbox{and} \ \
J_{\lambda,\varepsilon}'(u_{n})=(J_{\lambda,\varepsilon}^{T})'(u_{n})\rightarrow (J_{\lambda,\varepsilon}^{T})'(u^{T}_{\lambda})=0
$$
i.e.$$
J_{\lambda,\varepsilon}(u^{T}_{\lambda,\varepsilon})=c_{\lambda,\varepsilon}>0 \ \ \mbox{and} \ \
J_{\lambda,\varepsilon}'(u^{T}_{\lambda,\varepsilon})=0,
$$
showing that $u_{\lambda,\varepsilon}^{T}$ is the solution of \eqref{eq:equation} we were looking for. 

The first part of Theorem \ref{teorema1} is proved, with 
$$\lambda^{*}:=\max\{\lambda'(T),\lambda(T)\}, \quad u_{\lambda,\varepsilon}:=u_{\lambda,\varepsilon}^{T},
\quad \phi_{\lambda,\varepsilon} := \Phi_{\varepsilon}(u_{\lambda,\varepsilon})=\phi_{\varepsilon}(u_{\lambda,\varepsilon}).$$

For what concerns the positivity of the solutions, we observe that for every $u\in H^{1}(\mathbb R^{3})$,
the solution 
$\phi_{\varepsilon}(u)$ of the second equation in \eqref{eq:P} is nonnegative;
indeed this is easily seen by multiplying the second equation by $\phi_{\varepsilon}(u)^{-}:=\max\{-\phi_{\varepsilon}(u) , 0\}$ and integrating:
we arrive at
$$\int_{\mathbb R^{3}}|\nabla \phi_{\varepsilon}(u)^{-}|^{2}+\int_{\mathbb R^{3}}|\nabla \phi_{\varepsilon}(u)^{-}|^{4} \leq0$$
and the conclusion follows. Then, having $\phi_{\varepsilon}(u)\geq0$
we see analogously that the solution $u_{\lambda,\varepsilon}$  of
$$-\Delta u +u+\phi_{\varepsilon}(u) u = \lambda f(x,u) + |u|^{2^{*}-2}u$$
found above has to be nonnegative, being $f(x,t)=0$ for $t\leq0$.

Finally by similar computations as in the proof of Lemma
\ref{limitacao} we have that, fixed $\varepsilon>0$:
\begin{eqnarray*}\label{eq:}
0=\lim_{\lambda\to+\infty}  c_{\lambda,\varepsilon}&=& J_{\lambda,\varepsilon}(u_{\lambda,\varepsilon}) -\frac{1}{\theta}J_{\lambda,\varepsilon}'(u_{\lambda,\varepsilon})[u_{\lambda,\varepsilon}]\\
&\geq& \frac{\theta-2}{\theta}\|u_{\lambda,\varepsilon}\|_{H^{1}}^{2}.
\end{eqnarray*}
Then $\lim_{\lambda\to+\infty}u_{\lambda,\varepsilon}=0$ in $H^{1}(\mathbb R^{3})$ and by the continuity of the map $\Phi_{\varepsilon}$
defined in Lemma  \ref{lem:1}, we get also $\lim_{\lambda\to+\infty}\|\phi_{\lambda,\varepsilon}\|_{X}=0$.
As we have already said, the fact that $\lim_{\lambda\to+\infty}|\phi_{\lambda,\varepsilon}|_{\infty}=0$
follows by the continuous embedding of the space $X$ into $L^{\infty}(\mathbb R^{3})$.

Theorem \ref{teorema1} is completely proved.

\begin{remark}
Actually we have proved an additional property on the solution of \eqref{eq:equation}.
Indeed our method shows that, for every $T>0$ there exists a $\lambda^{*}=\lambda^{*}(T)$ such that
for all $\lambda>\lambda^{*}(T)$ and $\varepsilon\geq0$ there exists a solution $u_{\lambda,\varepsilon}$
of \eqref{eq:equation} with has norm less then $T$.
\end{remark}

\section{Proof of Theorem \ref{teorema2}}\label{sec:T2}
From now on we fix  the parameter $\overline\lambda$ greater then $\lambda^{*}=\max\{\lambda'(T),\lambda(T)\}$.
Our aim now is to show the behaviour of the solutions $u_{\overline\lambda,\varepsilon}$
with respect to $\varepsilon$.

%

Let us begin to show that $\{u_{\overline \lambda,\varepsilon}\}_{\varepsilon\geq0}$ is bounded.
We know that
\begin{equation}\label{eq:fato}
J_{\overline\lambda,\varepsilon}(u_{\overline \lambda,\varepsilon}) = c_{\overline \lambda,\varepsilon}\,,
\quad \ \ J_{\overline\lambda,\varepsilon}'(u_{\overline \lambda, \varepsilon})=0.
\end{equation}
Moreover, for this fixed $\overline\lambda>0$ we can invoke \eqref{ref4} and obtain that
\begin{equation}\label{eq:limita}
\forall \varepsilon>0:  \quad 0<c_{\overline\lambda,\varepsilon} \leq \frac{2^{*} - \theta}{2^{*}\theta}{S^{3/2}}.
\end{equation}

Then if we assume that $\lim_{\varepsilon\to0^{+}} \|u_{\overline\lambda,\varepsilon}\|_{H^{1}}=+\infty$,
exactly as in the proof of Lemma \ref{limitacao} (where we can replace $J_{\lambda,\varepsilon}^{T}$
with $J_{\overline\lambda,\varepsilon}$), by \eqref{eq:fato} we have:
\begin{equation*}
c_{\overline\lambda,\varepsilon} =J_{\overline\lambda,\varepsilon}(u_{\overline\lambda,\varepsilon})-
\dis\frac{1}{\theta}(J_{\overline\lambda,\varepsilon})'(u_{\overline\lambda,\varepsilon})[u_{\overline\lambda,\varepsilon}] + o_{\varepsilon}(1)
\geq\frac{\theta-2}{2\theta}\|u_{\overline\lambda,\varepsilon}\|_{H^{1}}^{2}  +o_{\varepsilon}(1)
\end{equation*}
which contrasts with \eqref{eq:limita}.
Then,  there exists  $u_{\overline\lambda,0}\in H^{1}(\mathbb R^{3})$ such that  up to subsequence,
$$u_{\overline\lambda, \varepsilon} \rightharpoonup u_{\overline\lambda,0}\quad \text{in  $H^{1}(\mathbb R^{3})$ \  as $\varepsilon\to0^{+}$}.$$
The fact that this convergence is strong, is done exactly in a straightforward way as in 
Section \ref{sec:final}.
This is based on the fact that the inequality in  \eqref{ref4} is true for every  $\varepsilon$.
Then the proof follows as before by replacing the limits in $n$ with limits with respect to $\varepsilon$.
In this way we first obtain  the strong convergence into $L^{2^{*}}(\mathbb R^{3})$,
and then
\begin{equation}\label{eq:strongconv}
\lim_{\varepsilon\to 0^{+}} u_{\overline\lambda,\varepsilon} = u_{\overline{\lambda},0}\quad \text{ in } H^{1}(\mathbb R^{3}).
\end{equation}
In particular we have that $u_{\overline\lambda,\varepsilon}^{2}\to u_{\overline\lambda, 0}^{2}$ in $L^{6/5}(\mathbb R^{3})$.

At this point we recall the following result  which we rewrite adapted to our notations.
\begin{lemma}(See  \cite[Lemma 3.2]{BK})\label{lem:kavian}
Let $f\in L^{6/5}(\mathbb R^{3}),  \{f_{\varepsilon}\}_{\varepsilon>0}\subset L^{6/5}(\mathbb R^{3})$ and
 assume
$\lim_{\varepsilon\to0^{+}}f_{\varepsilon}= f$ in $L^{6/5}(\mathbb R^{3})$.
Then
\begin{eqnarray*}
&\displaystyle\lim_{\varepsilon\to0^{+}}\phi_{\varepsilon}(f_{\varepsilon}) = \phi_{0}(f) \quad\text{in } D^{1,2}(\mathbb R^{3}),\\
&\displaystyle\lim_{\varepsilon\to 0^{+}}\varepsilon \phi_{\varepsilon}(f_{\varepsilon}) = 0\quad \text{ in } D^{1,4}(\mathbb R^{3}).
\end{eqnarray*}

\end{lemma}
In our case we have then
\begin{equation}\label{eq:p}
\phi_{\varepsilon}(u_{\overline\lambda,\varepsilon}) \to \phi_{0}(u_{\overline\lambda,0}) \quad \text{ in } D^{1,2}(\mathbb R^{3}), \quad 
\varepsilon \phi_{\varepsilon}(u_{\overline\lambda,\varepsilon}) \to 0\quad \text{ in } D^{1,4}(\mathbb R^{3}).
\end{equation}
To conclude the proof of Theorem \ref{teorema2}, let $v\in C^{\infty}_{c}(\mathbb R^{3})$ with $\textrm{supp}(v)\subset K$.
 We know that
\begin{equation}\label{eq:final1}
\langle u_{\overline\lambda,\varepsilon}, v\rangle_{H^{1}} +
\int_{K}\phi_{\varepsilon}(u_{\overline\lambda,\varepsilon})u_{\overline\lambda,\varepsilon} v=\int_{K}\overline\lambda f(x,u_{\overline\lambda,\varepsilon})v -\int_{K} |u_{\overline\lambda,\varepsilon}|^{2^{*}-2}u_{\overline\lambda,\varepsilon}v.
\end{equation}
We want to pass to the limit as $\varepsilon\to 0^{+}$ in the above identity. Let us see every term.

Of course
\begin{equation}\label{eq:final2}
\langle u_{\overline\lambda,\varepsilon}, v\rangle_{H^{1}} \to \langle u_{\overline\lambda,0}, v\rangle_{H^{1}}.
\end{equation}
Since $\phi_{\varepsilon}(u_{\overline\lambda,\varepsilon})\to \phi_{0}(u_{\overline\lambda,0})$ in $L^{6}(\mathbb R^{3}), u_{\overline\lambda,\varepsilon}\to u_{\overline \lambda,0}$ in $L^{12/5}(K)$ and $v\in L^{12/5}(K)$
we easily find
\begin{equation}\label{eq:final3}
\int_{K}\phi_{\varepsilon}(u_{\overline\lambda,\varepsilon})u_{\overline\lambda,\varepsilon} v
\to\int_{K} \phi_{0}(u_{\overline\lambda,0})u_{\overline\lambda,0} v.
\end{equation}
Moreover in a standard way we have also 
\begin{equation}\label{eq:final4}
\int_{K} f(x,u _{\overline\lambda,\varepsilon})v \to \int_{K} f(x,u _{\overline\lambda,0})v .
\end{equation}
and
\begin{equation}\label{eq:final5}
\int_{K} |u_{\overline\lambda,\varepsilon}|^{2^{*}-2}u_{\overline\lambda,\varepsilon}v\to
\int_{K} |u_{\overline\lambda,0}|^{2^{*}-2}u_{\overline\lambda,0}v.
\end{equation}

By \eqref{eq:final1}-\eqref{eq:final5} we deduce that
\begin{equation}\label{eq:finalissima}
\langle u_{\overline\lambda,0}, v\rangle +
\int_{K}\phi_{0}(u_{\overline\lambda,0})u_{\overline\lambda,0} v=\int_{K}\overline\lambda f(x,u_{\overline\lambda,0})v -\int_{K} |u_{\overline\lambda,0}|^{2^{*}-2}u_{\overline\lambda,0}v
\end{equation}
and this says that $u_{\overline\lambda,0}$ gives rise to a solution $(u_{\overline\lambda,0},\phi_{0}(u_{\overline\lambda,0}))$
 of the Schr\"odinger-Poisson system \eqref{eq:SP}.
 
 Then by setting 
 $$\phi_{\overline\lambda,\varepsilon}:=\phi_{\varepsilon}(u_{\overline{\lambda},\varepsilon}),\quad \phi_{\overline\lambda,0}:=\phi_{0}(u_{\overline{\lambda},0}),$$
  the proof of Theorem \ref{teorema2} follows by \eqref{eq:strongconv}, \eqref{eq:p} and \eqref{eq:finalissima}.

\begin{remark}
As a byproduct we get 
$I_{\varepsilon}(u_{\varepsilon})\to I_{0}(u_{0}) = \frac{1}{4}\int_{\mathbb R^{3}}|\nabla \phi_{0}(u_{\overline\lambda,0})|^{2}$ and consequently
we have the convergence of the mountain  pass energy levels,  $c_{\overline\lambda,\varepsilon}\to c_{\overline\lambda,0}$.
\end{remark}

\section{proof of Theorem \ref{teorema3}}\label{sec:T3}

In this section we study  the supercritical case, that is the problem
\begin{equation}\label{eq:Plambdap2}
\left\{
\begin{array}[c]{ll}
-\Delta u + u+\phi u = \lambda f(x,u)+|u|^{p-2}u & \ \mbox{in} \ \
\mathbb{R}^{3}, p>2^{*},\medskip\\
 -\Delta \phi - \varepsilon^{4}\Delta_4 \phi =u^{2}& \  \mbox{in} \ \ \mathbb{R}^{3},
\end{array}
 \right.
\end{equation}
under the same assumptions on $f$. 
We already know it is equivalent to consider the equation
\begin{equation}\label{eq:ridotta}
-\Delta u + u +\phi_{\varepsilon}(u) u =\lambda f(x,u)+|u|^{p-2}u  \quad \textrm{in } \mathbb R^{3}, \ \  p>2^{*}.
\end{equation}
To deal with  this case we consider a new nonlinearity  $g_{K}$, for $K>0$, given by
$$
g_{K}(x,t)=
\begin{cases}
 \lambda f(x,t)+|t|^{p-2}t  &\ \ \mbox{ if }\ \ |t|\leq K \\
 \lambda f(x,t)+K^{p-2^{*}}|t|^{2^{*}-2}t &\ \  \mbox{ if }\ \ |t|> K.
\end{cases}
$$
Once that 
$$
|g_{K}(x,t)|\leq \lambda f(x,t)+K^{p-2^{*}}|t|^{2^{*}-2}t \ \ \mbox{ if }\ \ t\in \mathbb R,
$$
we are in a position to apply Theorem \ref{teorema1} to  the equation
\begin{equation}\label{eq:ridottaK}
-\Delta u + u +\phi_{\varepsilon}(u) u =g_{K}(x,u) \quad \textrm{in } \mathbb R^{3},
\end{equation}
and then there exists a solution $u_{\lambda,\varepsilon,K}$ of \eqref{eq:ridottaK}.
  It is sufficient now to show that there exists $C>0$  independent on $\lambda$ and $K$ such that 
  \begin{equation}\label{eq:stimaMoser}
  |u_{\lambda,\varepsilon,K}|_{\infty}\leq C \|u_{\lambda,\varepsilon,K}\|_{H^{1}}.
  \end{equation}
  Indeed, since we know that
   $\|u_{\lambda,\varepsilon,K}\|_{H^{1}}\to 0$ as $\lambda \to \infty$, then, there is a $\lambda^{*}>0$ such that, for all $\lambda\geq \lambda^{*}$, $u_{\lambda,\varepsilon,K}$ is indeed a solution of   \eqref{eq:ridotta}. 
 
 However the proof of \eqref{eq:stimaMoser} can be obtained by repeating the arguments
  in the proof of  \cite[Theorem 1.1, pages 10-13]{FigueiredoPimenta}, taking into account
  the positivity of the solution of $-\Delta \phi -\varepsilon^{4}\Delta_{4} \phi =u^{2}$.

\end{document}